\documentclass{amsart}
\usepackage[cp1251]{inputenc}
\usepackage[english]{babel}
\usepackage{amsmath}
\usepackage{amssymb}
\usepackage{amsfonts}
\usepackage[english]{babel}
\usepackage[T2A]{fontenc}
\usepackage{amssymb, amsthm, amsmath}
\usepackage{graphicx}
\usepackage{float}
\usepackage[usenames]{color}
\usepackage{colortbl}
\usepackage{subfigure}
\newtheorem{theorem}{Theorem}

\newtheorem{lemma}{Lemma}
\newtheorem{propos}{Proposition}

\def\udcs{517.958} 

\renewcommand{\geq}{\geqslant}

\numberwithin{equation}{section}

\begin{document}
	
	UDC \udcs
	\thispagestyle{empty}

	\title[On components of stable connectivity\dots]{On components of stable connectivity of gradient-like diffeomorphisms of the 2-torus}

	\author{D.A. Baranov, E.V. Nozdrinova, O.V. Pochinka}
	
	\address{Denis Alekseevich Baranov,
		\newline\hphantom{iii} HSE University,
		\newline\hphantom{iii}Bolshaya Pecherskaya str., 25/12
		\newline\hphantom{iii} 603155, Nizhny Novgorod, Russia
	}
	\email{denbaranov0066@gmail.com}
	
	\address{Elena Vyacheslavovna Nozdrinova,
		\newline\hphantom{iii} HSE University,
		\newline\hphantom{iii}Bolshaya Pecherskaya str., 25/12
		\newline\hphantom{iii} 603155, Nizhny Novgorod, Russia
	}
	\email{maati@mail.ru}
	\address{Olga Vitalievna Pochinka,
		\newline\hphantom{iii} HSE University,
		\newline\hphantom{iii}Bolshaya Pecherskaya str., 25/12
		\newline\hphantom{iii} 603155, Nizhny Novgorod, Russia
	}
	\email{olga-pochinka@yandex.ru}
	
	\thanks{\sc Baranov D.A., Nozdrinova E.V., Pochinka O.V., On the components of stable connections of gradient-like diffeomorphisms of a 2-torus.}
	\thanks{\copyright \ D.A. Baranov, E.V. Nozdrinova, O.V. Pochinka \ 2025}
	\thanks{\rm This scientific work uses the results of the project ``Symmetry. Information. Chaos'', carried out within the framework of the HSE University Fundamental Research Program}
	
	\maketitle {\small
		\begin{quote}
			\noindent{\bf Abstract.} Gradient-like diffeomorphisms of a closed surface $M^2$ are characterized by a finite hyperbolic limit set and the absence of intersections of invariant manifolds of distinct saddle points. In the case where such diffeomorphisms $f_0, f_1:M^2\to M^2$ are isotopic, they are connected by some arc $\{f_t:M^2\to M^2, t\in [0,1]\}$ in the space of diffeomorphisms. If every diffeomorphism of the arc has a finite limit set and the arc is stable (does not change its qualitative properties under small perturbations) in the space of diffeomorphisms, then $f_0,f_1$ are said to be {\it stably connected}. Thus, the set of isotopic diffeomorphisms splits into components of stable connectivity, of which there may, in general, be infinitely many. For instance, it is known that gradient-like diffeomorphisms of the 2-sphere (both orientation-preserving and orientation-reversing) consist of a countable number of stable connectivity components. Moreover, belonging to a particular component is uniquely determined by the periodic data of the diffeomorphism. In the present paper, we consider gradient-like diffeomorphisms of the 2-torus that are not isotopic to the identity. We establish that the set of such diffeomorphisms splits into a finite number of stable connectivity components. For each isotopy class, we define the periodic data of the diffeomorphism, which uniquely determine membership in a given component.
			
			\medskip
			

			
		\end{quote}
	}

	\section{Introduction and statement of results}
	Let $M^n$ be a closed smooth\footnote{Wherever smooth objects are discussed, $C^\infty$-smoothness is assumed.} connected Riemannian $n$-manifold. Diffeomorphisms $\varphi_0,\varphi_1:M^n\to M^n$ are called {\it diffeotopic} if there exists a family of diffeomorphisms $\varphi_t:M^n\to M^n$ smoothly depending on the parameter $t\in[0,1]$, called an {\it arc}. According to \cite{NPT}, an arc $\varphi_t:M^n\to M^n$ is called {\it stable} if it is an interior point of an equivalence class under the following relation: arcs $\varphi_t$, $\varphi'_t$ are called {\it conjugate} if there exist homeomorphisms $h:~[0,1]~\to~[0,1], \, H_t:~M^n \to M^n$ such that $H_t \varphi_t = \varphi'_{h(t)} H_t$, and $H_t$ depends continuously on $t$.
	
	Morse-Smale diffeomorphisms (structurally stable diffeomorphisms with finite limit set) are said to belong to the same {\it stable connectivity} component if they can be connected in the space of diffeomorphisms by a stable arc $\varphi_t$ consisting of diffeomorphisms with finite limit set. In \cite{NPT}, it is established that all diffeomorphisms in such an arc are structurally stable, except for finitely many bifurcation diffeomorphisms $\varphi_{b_i}, i=1,\dots,q$, such that:
	\begin{itemize}
		\item[1)] the limit set of $\varphi_{b_i}$ contains exactly one non-hyperbolic periodic orbit, which is a saddle-node or a flip;
		\item[2)] $\varphi_{b_i}$ has no cycles;
		\item[3)] the invariant manifolds of all periodic points of $\varphi_{b_i}$ intersect transversally;
		\item[4)] the transition through $\varphi_{b_i}$ is a generically unfolded saddle-node or period doubling bifurcation, wherein the saddle-node point is non-critical (see, e.g., \cite{medved} for precise definitions).
	\end{itemize}

	It is known that the simplest Morse-Smale diffeomorphisms are the so-called {\it gradient-like diffeomorphisms}, characterized on the surface $M^2$ by the absence of intersections of the stable and unstable manifolds of different saddle points. The fact that isotopic gradient-like diffeomorphisms may not be connected by a stable arc was first noted in the paper \cite{blan}. The proof of this fact was based on the construction of a special filtration for a gradient-like diffeomorphism. However, the "sameness" of the filtrations is only a necessary, but not a sufficient condition for the existence of a stable arc between diffeomorphisms.

	As it turned out, a more promising approach to the stable classification of gradient-like diffeomorphisms is based on a close relationship between these diffeomorphisms and periodic transformations. Namely, it follows from the results of \cite{treh} that such diffeomorphisms $f:M^2\to M^2$ can be expressed as a composition
	\begin{equation}\label{comp}
		f = \phi_f \xi^1_f
	\end{equation} 
	of a periodic homeomorphism $\phi_f$ and the time-one map of a flow $\xi^t_f$ equivalent to the gradient flow of some Morse-Smale function, such that $\phi_f \xi^t_f = \xi^t_f \phi_f$ (see Proposition \ref{mn} below).  
	
	Recall that a {\it Morse-Smale function} is a Morse function whose invariant (stable and unstable) manifolds of its critical points intersect transversally.
	
	A homeomorphism $\phi :M^2 \rightarrow M^2$ is called {\it periodic} of order $m_\phi \in \mathbb{N}$ if $\phi^{m_\phi}=id$ and $\phi^j \neq id$ for any natural $j<m_\phi$. The classification problem of periodic homeomorphisms of surfaces up to topological conjugacy is completely solved in \cite{Nielsen37}, \cite{yoko}, \cite{yoko2}, \cite{yoko22}. 
	
	Using the relation \eqref{comp}, it was established in \cite{np2021} that gradient-like diffeomorphisms of the 2-sphere consist of a countable number of stable connectivity components. Moreover, membership in a particular component is uniquely determined by the periodic data of the diffeomorphism.
	
	In the present paper, we solve an analogous problem for orientation-preserving gradient-like diffeomorphisms of the 2-torus that are not isotopic to the identity.

	Every homeomorphism $\varphi:\mathbb T^2\to\mathbb T^2$ of the two-dimensional torus induces an action on the fundamental group, uniquely determined by a matrix $\varphi_\bigstar\in GL(2,\mathbb Z)$. Moreover, homeomorphisms $\varphi,\varphi':\mathbb T^2\to\mathbb T^2$ are isotopic if and only if $\varphi_\bigstar=\varphi'_\bigstar$. Thus, the isotopy classes of homeomorphisms are countable, and each such class is uniquely determined by a matrix $A\in GL(2,\mathbb Z)$. Matrices $A,A'\in SL(2,\mathbb Z)$ are called {\it similar over $\mathbb Z$} ($A\sim A'$) if there exists a matrix $B\in GL(2,\mathbb Z)$ such that $A'=BAB^{-1}$. For any matrix $B=\begin{pmatrix} a & c\\ b& d\end{pmatrix}\in GL(2,\mathbb Z)$, denote by $\widehat B:\mathbb T^2\to\mathbb T^2$ the algebraic diffeomorphism given by $$\widehat B(x,y)=(x,y)\begin{pmatrix} a & c\\ b& d\end{pmatrix}=(ax+by,cx+dy)\pmod 1.$$
	Set $$A_1=\begin{pmatrix} -1 & 0\\ 0& -1\end{pmatrix},\,A_2=\begin{pmatrix} -1 & -1\\ 1& 0\end{pmatrix},\,A_3=\begin{pmatrix} 0 & -1\\ 1& 0\end{pmatrix},\,A_4=\begin{pmatrix} 0 & -1\\ 1& 1\end{pmatrix}.$$ 
	
	\begin{propos}[\cite{BGPC}, Theorem 3]\label{perA} The algebraic diffeomorphisms $\widehat A_j,\,j\in\{1,\dots,4\}$ are periodic homeomorphisms, and any orientation-preserving periodic homeomorphism $\phi:\mathbb T^2\to\mathbb T^2$ not isotopic to the identity is topologically conjugate to exactly one of them. Moreover,
		$\phi\,\,\text{is topologically conjugate to}\,\,\widehat A_j\iff \phi_\bigstar\sim A_j.$
	\end{propos}
	
	Then from relation (\ref{comp}) it follows that the set of orientation-preserving gradient-like diffeomorphisms of the 2-torus not isotopic to the identity splits into 4 pairwise disjoint subsets $$\mathcal G_j=\left\{f:\mathbb T^2\to\mathbb T^2:\,f_\bigstar\sim A_j\right\},\,j\in\{1,\dots,4\}.$$
	
	Consider the following subset of $\mathcal G_j$: $$G_j=\left\{g:\mathbb T^2\to\mathbb T^2:\,g_\bigstar=A_j\right\}.$$
	Then any diffeomorphism $f\in\mathcal G_j$ is topologically conjugate to the diffeomorphism $g=\widehat B f\widehat B^{-1}\in G_j$, where $B\in GL(2,\mathbb Z)$ is a matrix such that $A_j=Bf_\bigstar B^{-1}$. Consequently, the problem of partitioning the set of isotopic diffeomorphisms into stable connectivity components can be solved for the sets $G_j$. For this, we will use the {\it characteristic} of the periodic homeomorphism $\widehat A_j$, which has the form $$n\,(n_1,\dots,n_k),$$ where $n$ is the period of $\widehat A_j$ and $n_1,\dots,n_k$ are the periods of its orbits smaller than $n$. According to \cite{BGPC}, 
	the homeomorphisms $\widehat A_j,j\in\{1,\dots,4\}$ have the following characteristics:
	$$\widehat A_1\sim 2\,(1,1,1,1),\,\widehat A_2\sim 3\,(1,1,1),\,\widehat A_3\sim 4\,(2,1,1),\,\widehat A_4\sim 6\,(3,2,1).$$
	By \cite{treh}, the periodic orbits of any diffeomorphism $g\in G_j$ have period $n$, except for $k$ orbits having periods $n_1,\dots,n_k$, where $n\,(n_1,\dots,n_k)$ is the characteristic of $\widehat A_j$. Let $i_1,\dots,i_k$ be the Morse indices (dimensions of unstable manifolds) of the orbits of $g$ with periods $n_1,\dots,n_k$, respectively. According to \cite{treh}, $i_l\neq 1$ if $n_l<\frac{n}{2}$. We will call the tuple $({n_1}_{_{i_1}},\dots,{n_k}_{_{i_k}})$ the {\it characteristic} of the diffeomorphism $g$. Denote by $$\langle {n_1}_{_{i_1}},\dots,{n_k}_{_{i_k}}\rangle$$ the subset of diffeomorphisms in the class $G_j$ having characteristic $({n_1}_{_{i_1}},\dots,{n_k}_{_{i_k}})$.
	
	The main result of this paper is the proof of the following theorem.
	
	\begin{theorem}\label{C_i} Each of the sets $G_j,\,j\in\{1,\dots,4\}$ consists of a finite number $r_j$ of stable connectivity components $G_j^0,\dots,G_j^{r_j-1}$, where:
		\begin{itemize}
			\item[1)] $r_1=1$, $G_1^0=G_1$;
			\item[2)] $r_2=4$, $G_2^0=\langle 1_0,1_0,1_0\rangle$, $G_2^1=\langle 1_0,1_0,1_2\rangle$,
			$G_2^2=\langle 1_0,1_2,1_2\rangle$,
			$G_2^3=\langle 1_2,1_2,1_2\rangle$;
			\item[3)] $r_3=3$, $G_3^0=\bigsqcup\limits_{i=0}^2\langle 2_i,1_0,1_0\rangle$, $G_3^1=\bigsqcup\limits_{i=0}^2\langle 2_i,1_0,1_2\rangle$,
			$G_3^2=\bigsqcup\limits_{i=0}^2\langle 2_i,1_2,1_2\rangle$;
			\item[4)] $r_4=4$, $G_4^0=\bigsqcup\limits_{i=0}^2\langle 3_i,2_0,1_0\rangle$, $G_4^1=\bigsqcup\limits_{i=0}^2\langle 3_i,2_0,1_2\rangle$,
			$G_4^2=\bigsqcup\limits_{i=0}^2\langle 3_i,2_2,1_0\rangle$,
			$G_4^3=\bigsqcup\limits_{i=0}^2\langle 3_i,2_2,1_2\rangle$.
		\end{itemize}
	\end{theorem}	
	
	\section{Necessary information about gradient-like diffeomorphisms of surfaces}
	\subsection{Morse-Smale diffeomorphisms} In this section, we briefly outline the necessary information about Morse-Smale diffeomorphisms (see, e.g., \cite{grin} for detailed information).
	
	Let $M^n$ be a closed smooth connected Riemannian $n$-manifold with metric $d$ and $f$ a diffeomorphism on $M^n$. For a diffeomorphism $f$, a point $x\in M^n$ is called {\it wandering} if there exists an open neighborhood $U_x$ of $x$ such that $f^{n}(U_x)\cap U_x=\emptyset$ for all $n\in \mathbb{N}$. Otherwise, the point $x$ is called {\it non-wandering}. 
	It follows directly from the definition that every point in a neighborhood $U_x$ is wandering, and hence the set of wandering points is open, while the non-wandering set is closed. 
	
	The set of all non-wandering points of a diffeomorphism $f$ is called the {\it non-wandering set} and is denoted by $\Omega_f$.
	
	The simplest examples of hyperbolic sets are primarily hyperbolic fixed points of a diffeomorphism, which can be classified as follows.
	Let $f:M^n\to M^n$ be a diffeomorphism and $f(p)=p$. The point $p$ is {\it hyperbolic} if and only if none of the eigenvalues of the derivative $D_pf$ has modulus equal to 1. If, moreover, all eigenvalues have modulus less than 1, then $p$ is called an {\it attracting, sink point or sink}; if all eigenvalues have modulus greater than 1, then $p$ is called a {\it repelling, source point or source}; otherwise, $p$ is called a {\it saddle point or saddle}.
	
	The hyperbolic structure at a fixed point $p$ implies the existence of its {\it stable} $W^s_p=\{x\in M^n:\,\lim\limits_{k\to +\infty}d(f^{k}(x),p)\to 0\}$ and {\it unstable} $W^u_p=\{x\in M^n:\,\lim\limits_{k\to +\infty}d(f^{-k}(x),p)\to 0\}$ manifolds, which are injective immersions of $\mathbb R^{n-q_p}$ and $\mathbb R^{q_p}$, respectively. Here $q_p$ is the number of eigenvalues of $D_pf$ with modulus greater than $1$, called the {\it Morse index} of $p$. The stable or unstable manifold is called an {\it invariant manifold} of this point.
	
	If $p$ is a periodic point of $f$ with period $m_p$, then applying the previous construction to $f^{m_p}$ yields a classification of hyperbolic periodic points analogous to the classification of hyperbolic fixed points.
	
	A diffeomorphism $f:M^n\to M^n$ is called a {\it Morse-Smale diffeomorphism} if
	
	1) the non-wandering set $\Omega_f$ consists of a finite number of hyperbolic points;
	
	2) the manifolds $W^s_p$, $W^u_q$ intersect transversally for any non-wandering points $p$, $q$.
	
	Thus, every non-wandering point $p$ of a Morse-Smale diffeomorphism $f:M^n\to M^n$ is a hyperbolic periodic point. Moreover, due to the finiteness of $\Omega_f$, the invariant manifolds $W^u_p,\,W^s_p$ are smooth submanifolds of the ambient manifold $M^n$ of dimensions $\dim\,W^u_p=q_p,\,\dim\,W^s_p=n-q_p$, respectively. A connected component of the set $W^u_p\setminus p$ ($W^s_p\setminus p$) is called an unstable (stable) {\it separatrix}. We say that a point $p\in\Omega_f$ of period $m_p$ has {\it positive (negative) orientation type} if the diffeomorphism $f^{m_p}|_{W^u_p}$ preserves (reverses) orientation. Due to the hyperbolicity of $\Omega_f$, it coincides with the limit set of $f$. For any subset $P\subset\Omega_f$, denote by $W^u_P\,(W^s_P)$ the union of the unstable (stable) manifolds of all points in $P$.
	
	Denote by $$\Omega_f^q,\,q\in\{0,\dots,n\}$$ the subset of the non-wandering set $\Omega_f$ of a Morse-Smale diffeomorphism $f:M^n\to M^n$ consisting of points with Morse index $q$. Set $C_q=|\Omega_f^q|$. The symbol $\beta_q(M^n) = \beta_q$ denotes the {\it $q$-th Betti number}, i.e., 
	$$\beta_q(M^n) = \mathrm{rank}\, H_q(M^n, \mathbb Z).$$ Let $\chi(M^n)$ denote the Euler characteristic of $M^n$, i.e., 
	\[\sum\limits_{q=0}^n(-1)^q \beta_q = \chi(M^n).\]
	\begin{propos}[Lefschetz-Hopf Theorem, \cite{ShubSull_Homology75},\cite{Smale_Morse60}]\label{lefschetz} For any Morse-Smale diffeomorphism $f:M^n\to M^n$, the following relations hold:
		\begin{align}\label{beti}
			& C_0 \geq \beta_0, & C_1-C_0 \geq \beta_1 - \beta_0,\,\,\,\,\,\,\,\, & C_2 - C_1 + C_0 \geq \beta_2 -\beta_1 + \beta_0,\,\,\dots\,\,,
		\end{align}
		\begin{equation}\label{sum}
			\sum\limits_{q=0}^n(-1)^q C_q= \chi(M^n).
		\end{equation}
	\end{propos}
	
	A compact $f$-invariant set $A\subset M^n$ is called an {\it attractor} of $f:M^n\to M^n$ if it has a compact neighborhood $U_{A}$ such that $f(U_{A})\subset {\rm int}\,(U_{A})$ and $A = \underset{k\geqslant 0}{\bigcap} f^{k}(U_{A}).$ The neighborhood $U_{A}$ is called a {\it trapping neighborhood}. A {\it repeller} is defined as an attractor for $f^{-1}$.
	
	Let $f:M^n\to M^n$ be a Morse-Smale diffeomorphism, $\tilde \Omega_f$ the set of its saddle points, and $\Sigma\subset\tilde \Omega_f$ a subset (possibly empty) such that $\left(cl(W^u_\Sigma)\setminus W^u_\Sigma\right)\subset\Omega^0_f$. Set $$A_\Sigma=\Omega^0_f\cup W^u_\Sigma,\,R_\Sigma=\Omega^n_f\cup W^s_{\tilde\Omega_f\setminus\Sigma}.$$ 
	
	\begin{propos}[\cite{GrZhuMePo}, Theorem 1.1]\label{<n-2} The set $A_\Sigma$ $(R_\Sigma)$ is an attractor (repeller) of the Morse-Smale diffeomorphism $f:M^n\to M^n$, and it is connected if the topological dimension of $R_\Sigma$ $(A_\Sigma)$ does not exceed $n-2$.
	\end{propos}
	
	\subsection{Gradient-like diffeomorphisms of a surface} A Morse-Smale diffeomorphism $f:M^n\to M^n$ is called {\it gradient-like} if $W^s_{p} \cap W^u_{q}\neq \emptyset $ for distinct points $p, q \in \Omega_f$ implies $\dim\, W^u_{p}< \dim\, W^u_{q}.$ In dimension $n=2$, the set of gradient-like diffeomorphisms coincides with the set of Morse-Smale diffeomorphisms whose saddle separatrices do not intersect.
	
	Let $M^2$ be a closed connected orientable surface.
	For any homeomorphism $\phi:M^2 \rightarrow M^2$, denote by $Per_\phi$ the set of its periodic points, by $P_\phi$ the set of periods of periodic points, and by $Fix_\phi$ the set of its fixed points. From the definition of a periodic homeomorphism $\phi$ of order $m_\phi$, it follows that $P_\phi$ is finite and consists of divisors of $m_\phi$. Set $\tilde P_\phi=P_\phi \backslash\{m_\phi\}$. For any $l\in\tilde P_\phi$, denote by $B_\phi^l$ the set of points of period $l$. Set $B_\phi=\bigsqcup\limits_{l\in \tilde P_\phi} B_\phi^l$. 
	
	\begin{propos}[\cite{Nielsen37}, \cite{yoko}]\label{data}
		Let $\phi :M^2 \rightarrow M^2$ be an orientation-preserving periodic homeomorphism of order $m_\phi$. Then $B_\phi^l$ is finite for any $l\in \tilde P_{\phi}$.	
	\end{propos}	
	
	\begin{propos}[\cite{treh}]\label{mn} Let $f:M^2 \rightarrow M^2$ be an orientation-preserving gradient-like diffeomorphism. Then there exist an orientation-preserving periodic homeomorphism $\phi_f:M^2 \rightarrow M^2$ of order $m_{f}$ and a flow $\xi^t_f:M^2 \rightarrow M^2$ equivalent to the gradient flow of a Morse-Smale function such that $\phi_f\xi^t_f=\xi^t_f\phi_f$ and $f=\phi_f\xi^1_f$. Moreover,
		\begin{enumerate}
			\item[1)] $B_{\phi_f}\subset\Omega_f=Fix_{\xi^1_f}$ and the invariant manifolds of the periodic points of $f$ coincide with the invariant manifolds of the fixed points of the flow $\xi^t_f$; 
			\item[2)] $f|_{\Omega_f}=\phi_f|_{\Omega_f}$ and $\tilde P_{\phi_f} \subset P_f \subset P_{\phi_f}$;
			\item[3)] $m_f$ is the smallest natural number $m$ such that $\Omega_f$ consists of fixed points of $f^m$ and all saddle points have positive orientation type; 
			\item[4)] the period of any saddle separatrix equals $m_f$.
		\end{enumerate}
	\end{propos}
	
	\begin{propos}[\cite{np2021}]\label{sad-no}
		Let $f:M^2\to M^2$ be a gradient-like diffeomorphism, and suppose there exist pairwise disjoint closed 2-disks $$D,f(D),\dots,f^{m-1}(D),\,m\in\mathbb N$$ such that $f^m(D)\subset {\rm int}\, D$, and $f^m$ has a fixed sink in $D$. Then there exists a stable arc connecting $f$ to a gradient-like diffeomorphism $\tilde f$ that coincides with $f$ outside the disks, and $\tilde f^m$ has exactly one non-wandering point in $D$ — a fixed sink.
	\end{propos}	
	
	\section{Stable connectivity components of the sets $G_1,G_2$}
	The fact that any two diffeomorphisms in $G_1$ are connected by a stable arc is established in \cite[Theorem 1]{NoPoNo}. The phase portrait of a diffeomorphism in $G_1$ to which all other diffeomorphisms in this set are connected by a stable arc is shown in Figure \ref{modelG1}.  
	\begin{figure}[h]		
		\centerline{\includegraphics[width=8 cm]{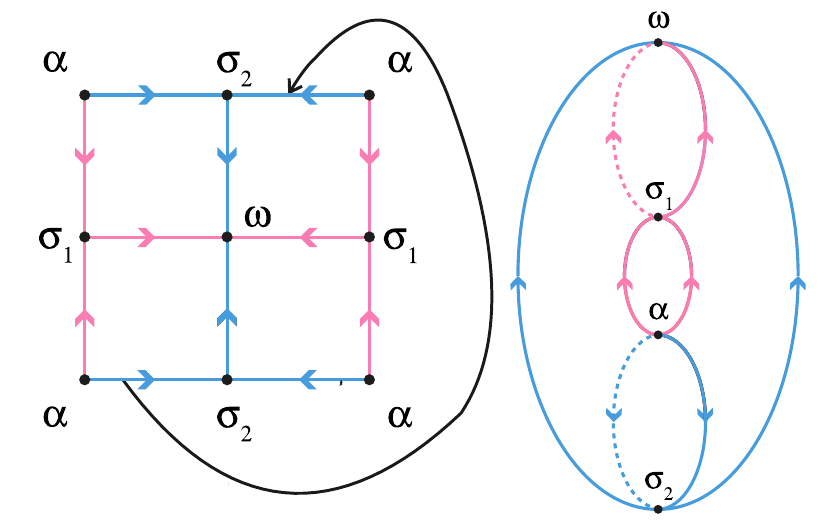}}		
		\caption{\small Phase portrait of a diffeomorphism in $G_{1}$}\label{modelG1}	
	\end{figure}
	
	The existence of four stable connectivity components
	$G_2^0=\langle 1_0,1_0,1_0\rangle$, $G_2^1=\langle 1_0,1_0,1_2\rangle$,
	$G_2^2=\langle 1_0,1_2,1_2\rangle$,
	$G_2^3=\langle 1_2,1_2,1_2\rangle$ in $G_2$ is established in \cite{BaPo25}. Examples of phase portraits of diffeomorphisms $g_{i}\in G^i_2,\,i\in\{0,1,2,3\}$ to which all other diffeomorphisms in $G^i_2$ are connected by a stable arc are shown in Figure \ref{modelG2}. 
	\begin{figure}[H]
		\begin{minipage}[h]{0.47\linewidth}
			\center{\includegraphics[width=1\linewidth]{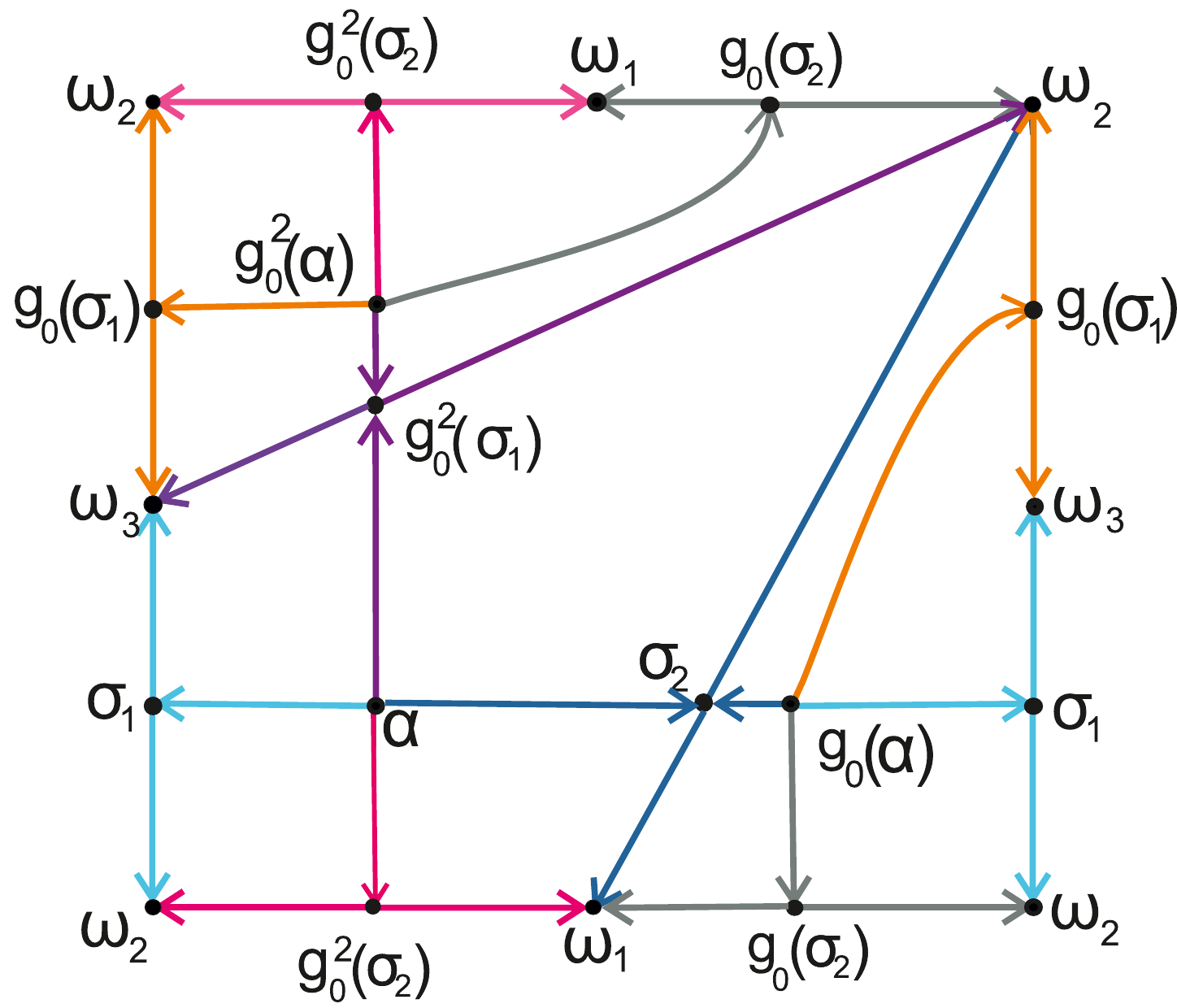}} (a) \\
		\end{minipage}
		\hfill
		\begin{minipage}[h]{0.47\linewidth}
			\center{\includegraphics[width=1\linewidth]{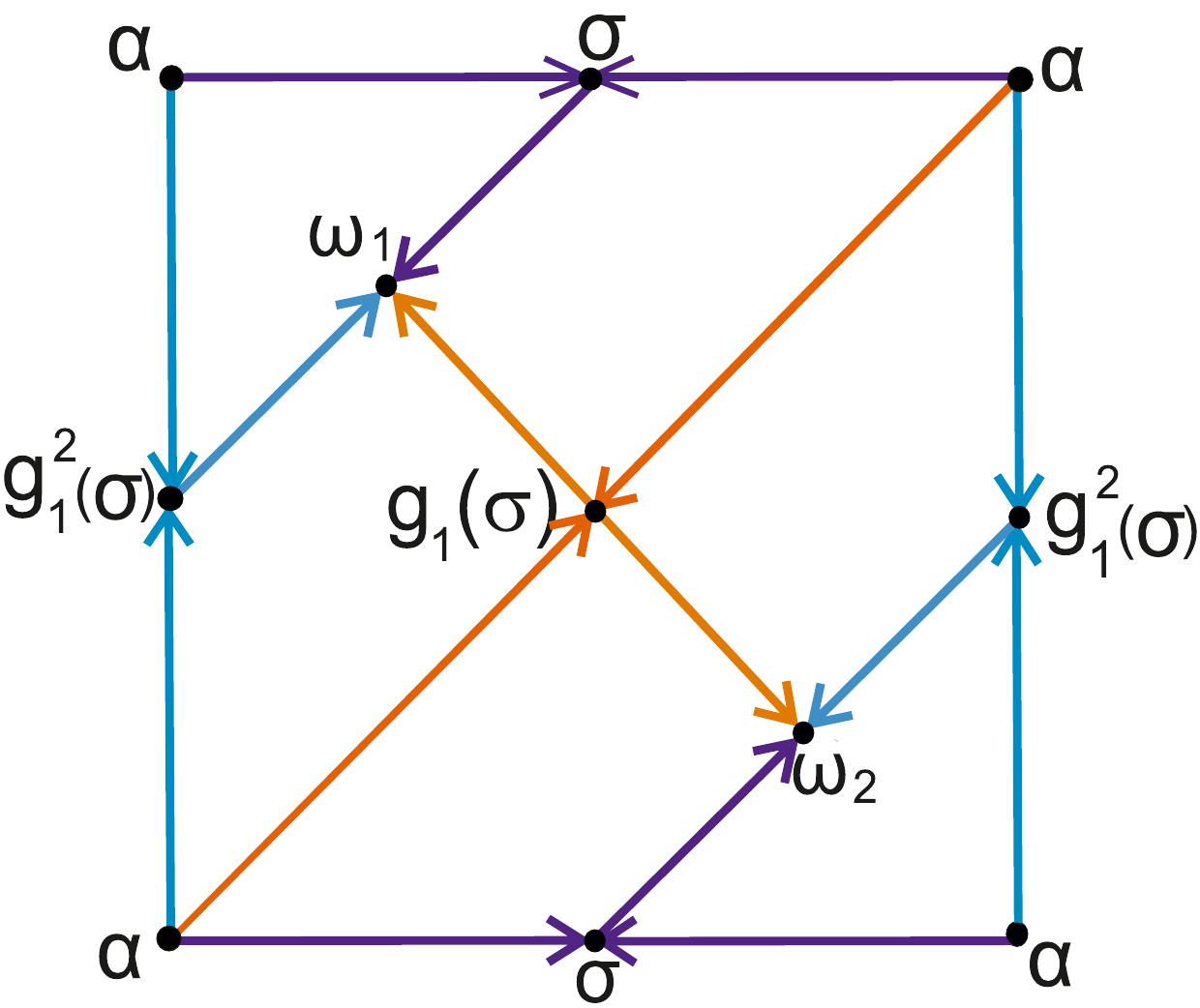}} \\(b)
		\end{minipage}
		\vfill
		\begin{minipage}[h]{0.47\linewidth}
			\center{\includegraphics[width=1\linewidth]{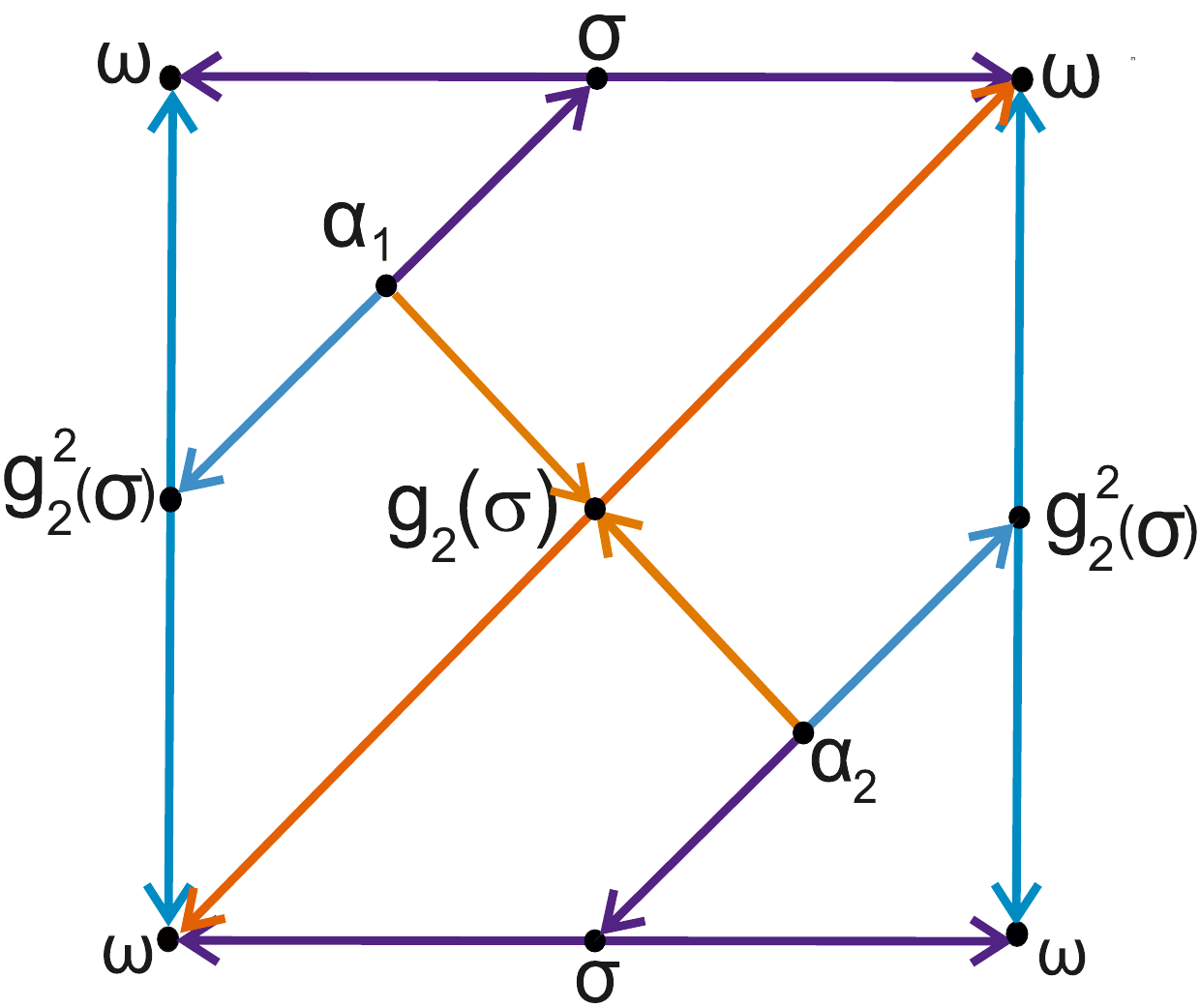}} (c) \\
		\end{minipage}
		\hfill
		\begin{minipage}[h]{0.47\linewidth}
			\center{\includegraphics[width=1\linewidth]{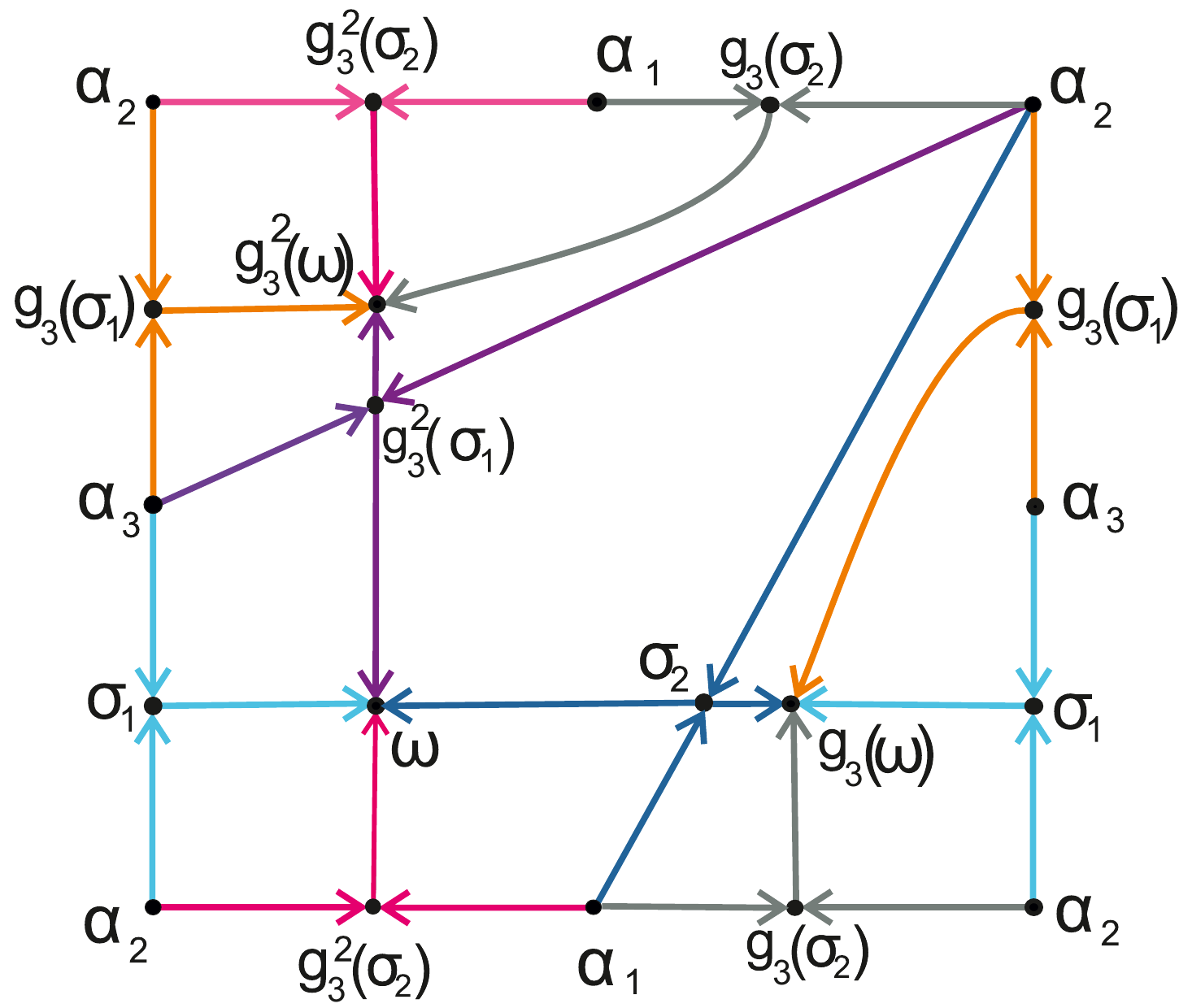}} (d) \\
		\end{minipage}
		\caption{Phase portraits of diffeomorphisms: (a) $g_{0}\in G^0_2$, (b) $g_{1}\in G_2^1$, (c) $g_{2}\in G_2^2$, (d) $g_{3}\in G_2^3$}
		\label{modelG2}
	\end{figure}
	
	Thus, it remains to prove items 3) and 4) of Theorem \ref{C_i}. The proofs in both cases consist of three parts:
	
	I) description of the dynamics of the {\it simplest} diffeomorphism in $G_j^i,\,j\in\{3,4\},\,i\in\{0,\dots, r_j-1\}$ denoted by $g_{i}$, defined by the condition 
	\begin{equation}\label{min}
		|\Omega_{g_{i}}|=\min_{g\in G_j^i}|\Omega_g|;
	\end{equation}
	
	II) proof of the absence of a stable arc connecting the simplest diffeomorphisms $g_i,g_{\tilde i}\in G^i_j$ for $i\neq\tilde i$;
	
	III) construction of a stable arc connecting an arbitrary diffeomorphism $g\in G_j^i$ to the simplest diffeomorphism $g_i \in G_j^i$. 
	
	\section{Dynamics of the simplest diffeomorphisms}
	In this section, we describe the dynamics of the simplest diffeomorphism in each set $G_j^i,\,j\in\{3,4\},\,i\in\{0,\dots, r_j-1\}$. Due to the similarity of the construction ideas in all cases, we will give a comprehensive proof for the diffeomorphism $g_{2}\in G_4^2$ and only present the phase portraits for the remaining simplest diffeomorphisms. 
	
	\begin{lemma}\label{g10} The simplest diffeomorphism $g_{2}\in G_4^2$ has the following properties (see Fig. \ref{picg1}):
		\begin{enumerate}
			\item[1)] $\Omega_{g_{2}}$ consists of one fixed sink $\omega$, one source orbit $\alpha,g_{2}(\alpha)$ of period $2$, and one saddle orbit $\sigma,g_{2}(\sigma),g_{2}^2(\sigma)$ of period $3$;
			\item[2)] the set $K=W^u_\sigma\cup\omega$ is a non-contractible knot on the torus $\mathbb T^2$;
			\item[3)] the knots $K,g_{2}(K), g_{2}^2(K)$ have homotopy types $\pm\langle 1,0\rangle$, $\pm\langle 0,-1\rangle$, $\mp\langle 1,1\rangle$. 
		\end{enumerate}
		Any two simplest diffeomorphisms in $G_4^2$ are topologically conjugate.	
	\end{lemma}
	\begin{proof} We prove each item of the lemma in sequence.
		
		1) Since $G_4^2=\bigsqcup\limits_{i=0}^2\langle 3_i,2_2,1_0\rangle$, any diffeomorphism in $G^2_4$ necessarily has in its non-wandering set a fixed sink, a source orbit of period $2$, and some orbit of period $3$. The minimality condition \eqref{min} and formula \eqref{sum} lead to the following structure of the non-wandering set of the simplest diffeomorphism $g_2\in G_4^2$:
		$$\Omega_{g_2}=\{\omega,\alpha,g_{2}(\alpha),\sigma,g_{2}(\sigma),g_{2}^2(\sigma)\},$$
		where $\omega,\alpha,\sigma$ are a sink, a source, and a saddle, respectively.
		
		2) Since $\Omega_{g_{2}}$ contains a unique sink $\omega$, we have ${\rm cl}\,W^u_\sigma\setminus W^u_\sigma=\omega$. Hence, the set $K=W^u_\sigma\cup\omega$ is a knot on $\mathbb T^2$. We show that it is non-contractible. Assuming the contrary, we obtain that $K$ bounds a 2-disk $D\subset\mathbb T^2$. By item 4) of Proposition \ref{mn}, the map $g^3_2|_{K}$ reverses orientation. On the other hand, $g^3_2$ preserves orientation, and consequently $g_2^3(D)=\mathbb T^2\setminus D$. This contradicts the fact that the torus is not the union of two disks glued along their boundary.
		
		Thus, $K$ is a non-contractible knot.
		
		3) Let the knot $K$ have homotopy type $\langle a,b\rangle$. Since $g_{{2}\star}=A_4$, the knots $K,g_{2}(K),g_{2}^2(K)$ have the following homotopy types:
		\begin{itemize}
			\item $\langle K\rangle=\langle a,b\rangle$;
			\item $\langle g_{2}(K)\rangle=\langle a,b\rangle A_4=\langle b,b-a\rangle$;
			\item $\langle g_{2}^2(K)\rangle=\langle b,b-a\rangle A_4=\langle b-a,-a\rangle$.
		\end{itemize} 
		By item 2) proved above, the knots $K,g_{2}(K)$ are non-contractible and have a unique intersection point $\omega$. According to \cite{Rol}, the determinant of the matrix $\begin{pmatrix} a & b\\ b& b-a\end{pmatrix}$ must in this case have absolute value $1$. A direct computation yields the following possible homotopy types for $\langle a,b\rangle$: $\pm\langle 1,0\rangle$, $\pm\langle 1,1\rangle$, $\pm\langle 0,1\rangle$. Moreover, if $\langle K\rangle=\langle 1,0\rangle$, then 
		$\langle g_2(K)\rangle=\langle 0,-1\rangle,$ $\langle g^2_2(K)\rangle=\langle -1,-1\rangle$, $\langle g_2^3(K)\rangle=\langle -1,0\rangle,$ $\langle g^4_2(K)\rangle=\langle 0,1\rangle$, $\langle g^5_2(K)\rangle=\langle 1,1\rangle.$ Placing the knots $K,g_{2}(K),g_{2}^2(K)$ on $\mathbb T^2$ according to their homotopy types, we obtain the phase portrait of $g_{2}\in G^2_3$ (see Fig. \ref{picg1}).
		\begin{figure}[h]		
			\centerline{\includegraphics[width=7 cm]{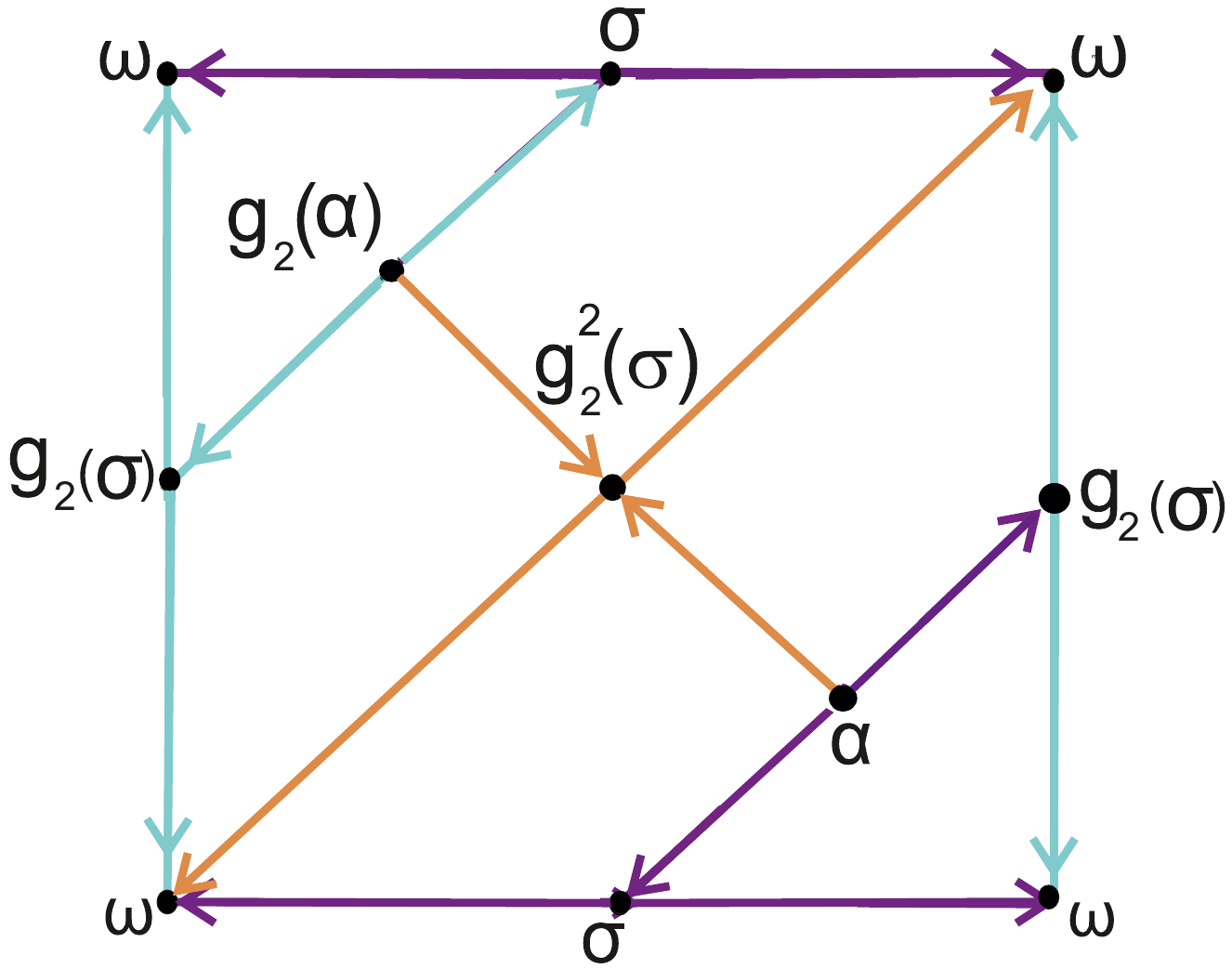}}		
			\caption{\small Phase portrait of $g_{2}\in G^2_4$}\label{picg1}	
		\end{figure}
		
		It remains to show that any two simplest diffeomorphisms in $G_4^2$ are topologically conjugate. To this end, we use the tricolor graph, whose equivalence class is a complete invariant of gradient-like diffeomorphisms (see \cite{treh}). To construct the graph, we color all stable (unstable) saddle separatrices blue (red). In each region complementary to the closure of the saddle separatrices, we choose one $g_2$-invariant curve and color it green. As a result, $\mathbb T^2$ is partitioned into triangular regions by colored curves. We associate a vertex to each such region and connect two vertices by an edge of the corresponding color if the regions share a boundary of that color. The resulting graph $\Gamma_{g_{2}}$ is called the {\it tricolor graph} of $g_{2}$ (see Fig. \ref{gr1}). 
		\begin{figure}[h]		
			\centerline{\includegraphics[width=11 cm]{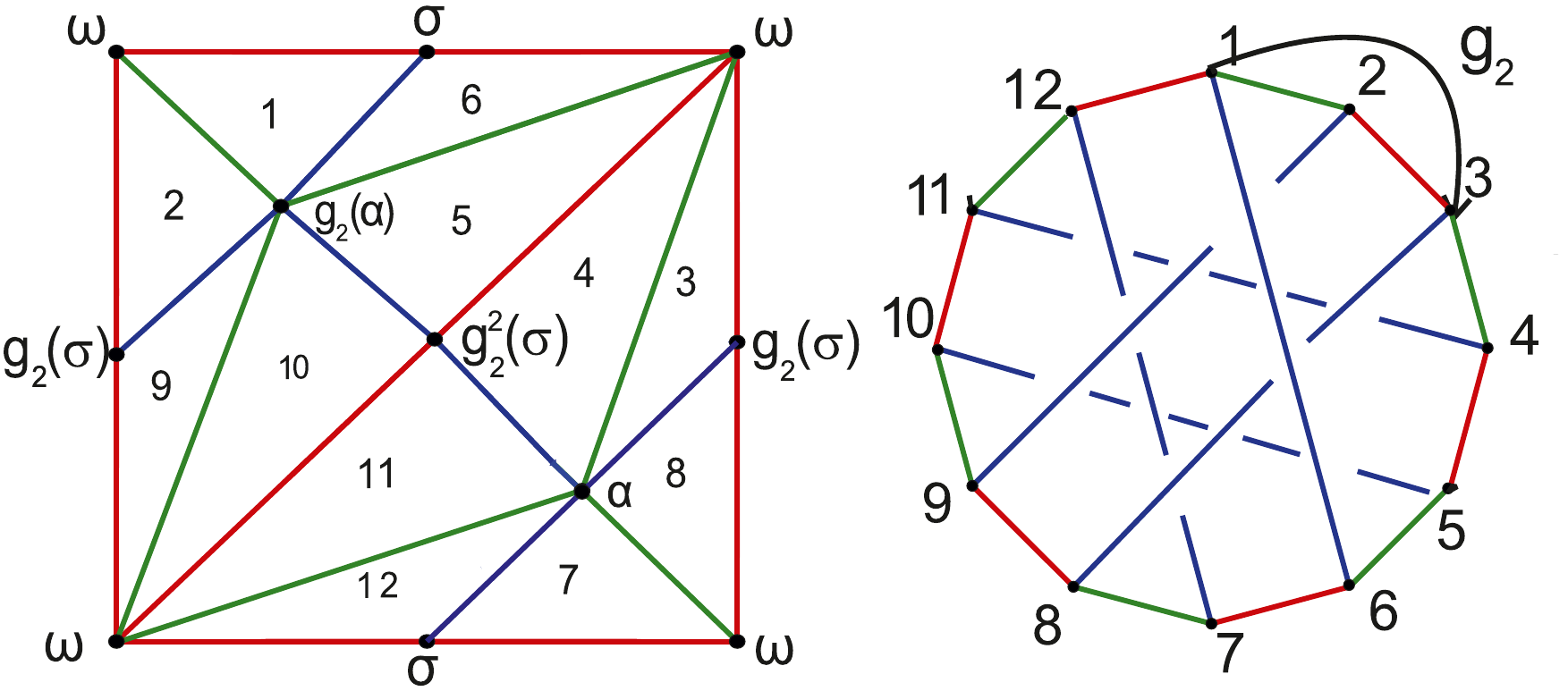}}		
			\caption{\small Tricolor graph of $g_{2}\in G^2_4$}\label{gr1}	
		\end{figure}
		
		By construction, all vertices of $\Gamma_{g_{2}}$ form a single red-green cycle, and $g_{2}$ induces a permutation of the vertices consisting of a rotation of this cycle by $\frac{\pi}{3}$. Two tricolor graphs are {\it equivalent} if there exists an isomorphism between them preserving edge colors and conjugating the permutations. The equivalence class of the graph is a complete invariant of topological conjugacy of the corresponding gradient-like diffeomorphism. Moreover, for any such graphs there exists an isomorphism commuting with the rotation, which completes the proof.	
	\end{proof}
	
	Figure \ref{4} shows the phase portraits of the simplest diffeomorphisms $g_i\in G_4^i$.
	\begin{figure}[H]
		\center{\includegraphics[width=1\linewidth]{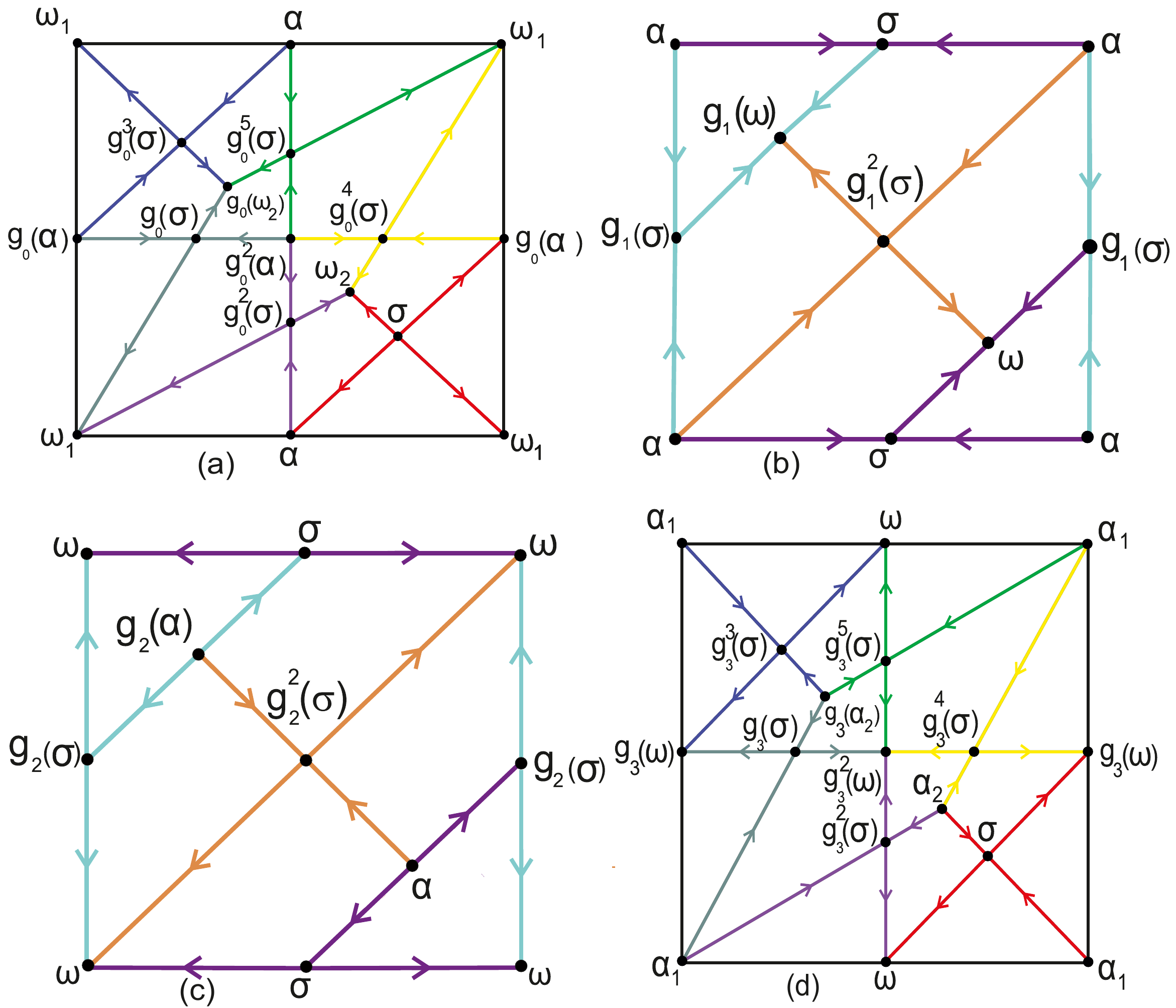}}
		\caption{Phase portraits of diffeomorphisms: (a) $g_{0}\in G^0_4$, (b) $g_{1}\in G_4^1$, (c) $g_{2}\in G_4^2$, (d) $g_{3}\in G_4^3$}
		\label{4}
	\end{figure}
	
	Figure \ref{3} shows the phase portraits of the simplest diffeomorphisms $g_i\in G_3^i$.
	\begin{figure}[H]
		\center{\includegraphics[width=1\linewidth]{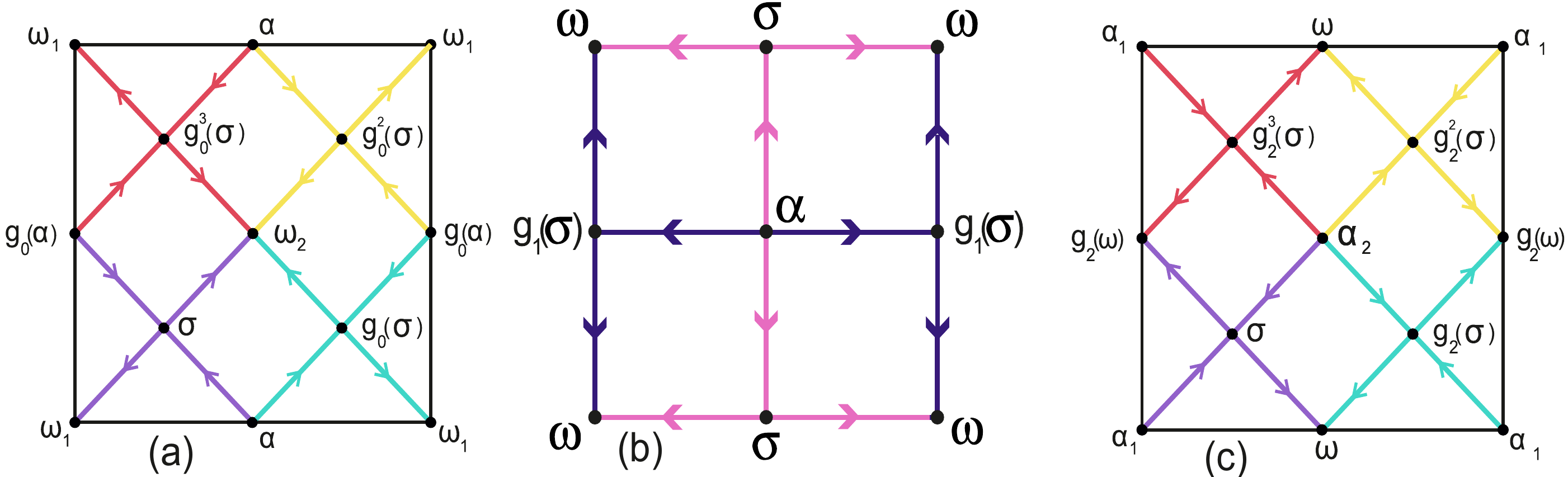}}
		\caption{Phase portraits of diffeomorphisms: (a) $g_{0}\in G^0_3$, (b) $g_{1}\in G_3^1$, (c) $g_{2}\in G_3^2$}
		\label{3}
	\end{figure}
	
	\section{Absence of a stable arc connecting distinct simplest diffeomorphisms}	
	
	The absence of a stable arc connecting distinct isotopic simplest diffeomorphisms is an immediate consequence of the following lemma.  
	\begin{lemma}\label{step1} For any stable arc $f_t,\,t\in[0,1]$ such that $f_0\in G_j^i$, we have $f_1\in G_j^i$.
	\end{lemma}
	\begin{proof} Let $P_{f_t}$ be the set of periodic points of the gradient-like diffeomorphism $f_t$ with periods $1$ and $2$. From the definition of the classes $G_j^i$, for each $f_t$ there exists an index $i_t\in\{0,\dots,r_j\}$ such that $f_t\in G_j^{i_t}$, and $i_{t_1}=i_{t_2}$ if and only if the maps $f_{t_1}|_{P_{f_{t_1}}}$ and $f_{t_2}|_{P_{f_{t_2}}}$ are topologically conjugate. We show that $i_t=i_0$ for any $t\in[0,1]$.
		
		Since all points in $P_{f_{0}}$ have period less than $\frac{m_{f_0}}{2}$, they are all nodal. For a point $p_0\in P_{f_{0}}$ to participate in a saddle-node bifurcation, there must exist an $f^{per(p_0)}_0$-invariant curve (center manifold) containing this point in its interior. However, in a neighborhood of $p_0$, the map $f^{per(p_0)}_0$ is topologically conjugate to a composition of a homothety and a rotation of the plane by an angle less than $\pi$, and therefore such a curve does not exist.
		
		Period-doubling bifurcations involve points of adjacent indices whose periods differ by a factor of two; hence at least one of these points is a saddle. By Proposition \ref{mn}, any saddle point of $f_0$ has period $m_{f_0}$ or $\frac{m_{f_0}}{2}$, while $p_0$ has period less than $\frac{m_{f_0}}{2}$, so no such saddle point exists.
		
		Since any bifurcation on a stable arc preserves the dynamics (up to topological conjugacy) outside a neighborhood of the bifurcation points, the maps $f_{t}|_{P_{f_{t}}}$ and $f_{0}|_{P_{f_{0}}}$ are topologically conjugate for any $t\in[0,1]$.
	\end{proof}
	
	\section{Construction of a stable arc connecting an arbitrary diffeomorphism to the simplest one}
	
	\begin{lemma}\label{step2} Any diffeomorphism $g \in G^i_j$ is connected by a stable arc to some simplest diffeomorphism $g_{i}\in G_j^i$.
	\end{lemma}
	\begin{proof} Due to the similarity of the construction ideas in all cases, we give a comprehensive proof for the diffeomorphism $g\in G_4^2$. 		
		
		According to Lemma \ref{g10}, the set of orbits of period less than $6$ consists of one fixed sink $\omega$, one source orbit of period $2$ $\alpha, g(\alpha)$, and one orbit of period $3$. The period-$3$ orbit may be saddle or nodal. We now consider two cases separately: 1) the period-$3$ orbit is saddle; 2) the period-$3$ orbit is nodal. 
		
		1) Suppose $g$ has a saddle point $\sigma$ of period $3$. We show that in this case $g$ is connected by a stable arc to the simplest diffeomorphism.
		
		Since $g_{\bigstar}=A_4$, similarly to Lemma \ref{g10} one can prove that the knots $K=W^u_\sigma\cup\omega,{g}(K),{g}^2(K)$ have homotopy types $\pm\langle 1,0\rangle$, $\mp\langle 1,1\rangle$, $\pm\langle 0,1\rangle$. Moreover, by Proposition \ref{sad-no}, the set $A=K\cup{g}(K)\cup{g}^2(K)$ is an attractor of ${g}$. Its trapping neighborhood $U_{A}$ partitions $\mathbb T^2$ into $2$-disks $D_1 \sqcup D_2$ such that ${g}^{-1}(D_1) \subset {\rm int}\, D_2, {g}^{-1}(D_2) \subset {\rm int}\, D_1$ (see Fig. \ref{disk}).
		\begin{figure}[h]	
			\centerline{\includegraphics[width=7 cm]{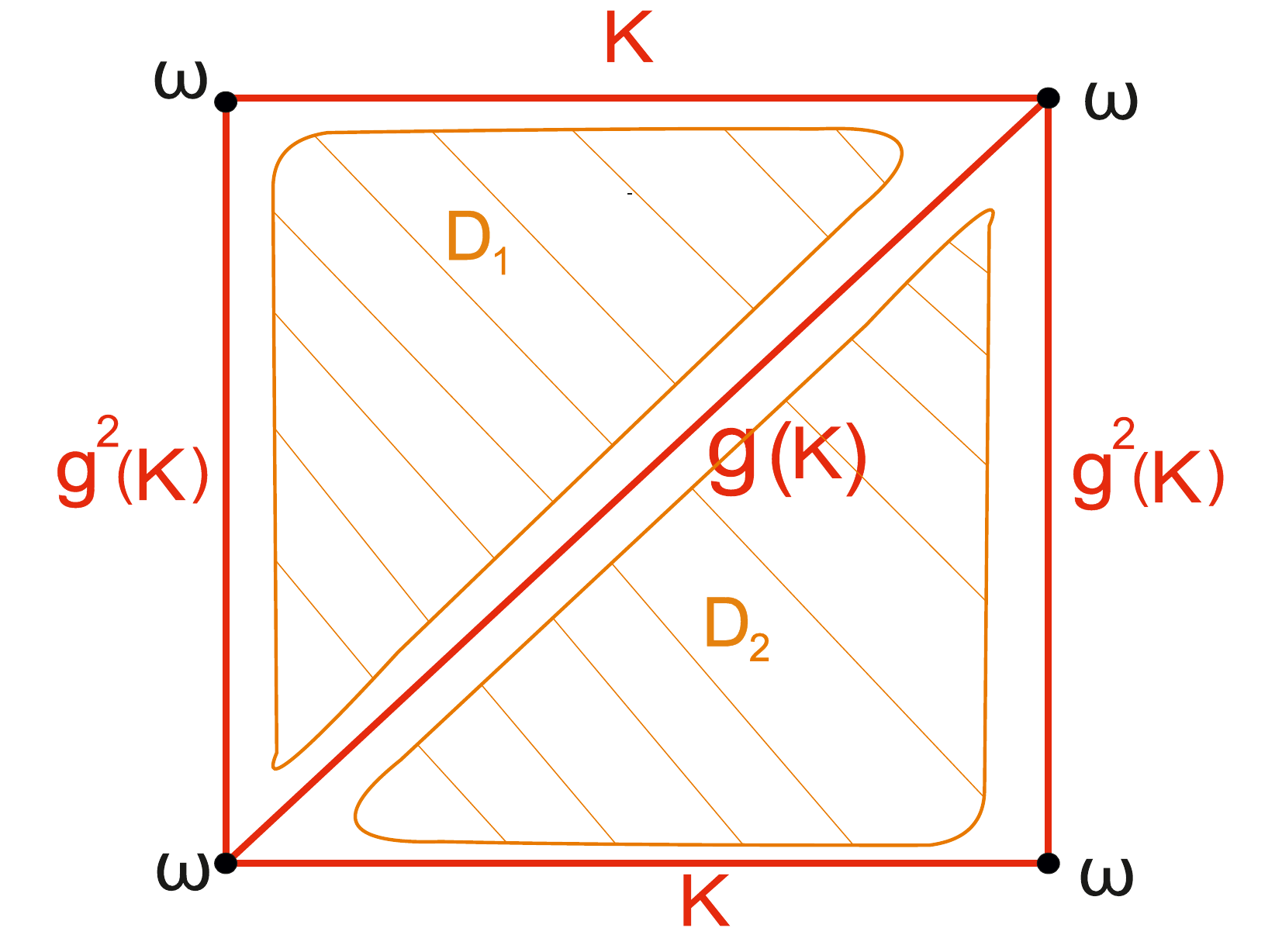}}		
			\caption{\small Illustration for case 1)}\label{disk}	
		\end{figure}	
		Thus, ${g}^{-2}(D_i)\subset {\rm int}\, D_i$, $i=1,2$, so by the Brouwer fixed point theorem, ${g}^{-2}$ has at least one fixed point in $D_i$. Since ${g}^{-2}$ has exactly two fixed points outside $A$, we have $\alpha\in D_1, {g}(\alpha)\in D_2$. By Proposition \ref{sad-no}, there exists a stable arc connecting ${g}$ to the simplest diffeomorphism ${g}_2\in G^2_4$. 
		
		2) Assume that $g$ has a sink $\tilde\omega$ of period $3$ (the proof for a source is analogous). We show that in this case $g$ is connected by a stable arc to a diffeomorphism satisfying case 1). 
		
		By Proposition \ref{<n-2}, the set $\Omega^0_{g}\cup W^u_{\Omega_{g}^1}$ is connected; hence there exists a saddle point $\tilde\sigma$ such that ${\rm cl}\,W^u_{\tilde\sigma}\setminus W^u_{\tilde\sigma}=\tilde\omega\sqcup\hat\omega$, where $\hat\omega$ is either the fixed sink $\omega$ or a sink of period $6$. Since the period of $\tilde\sigma$ is $6$, in both cases (see Fig. \ref{case2}) the sink $\tilde\omega$ belongs to the closure of the unstable manifolds $W^u_{\tilde\sigma},W^u_{g^3(\tilde\sigma)}$. A flip bifurcation involving the orbits of $\tilde\omega,\tilde\sigma$ connects $g$ to a diffeomorphism $\tilde g$ having a saddle point $\sigma$ of period $3$ (for details, see \cite{np2021}). 
		\begin{figure}[h]	
			\centerline{\includegraphics[width=10 cm]{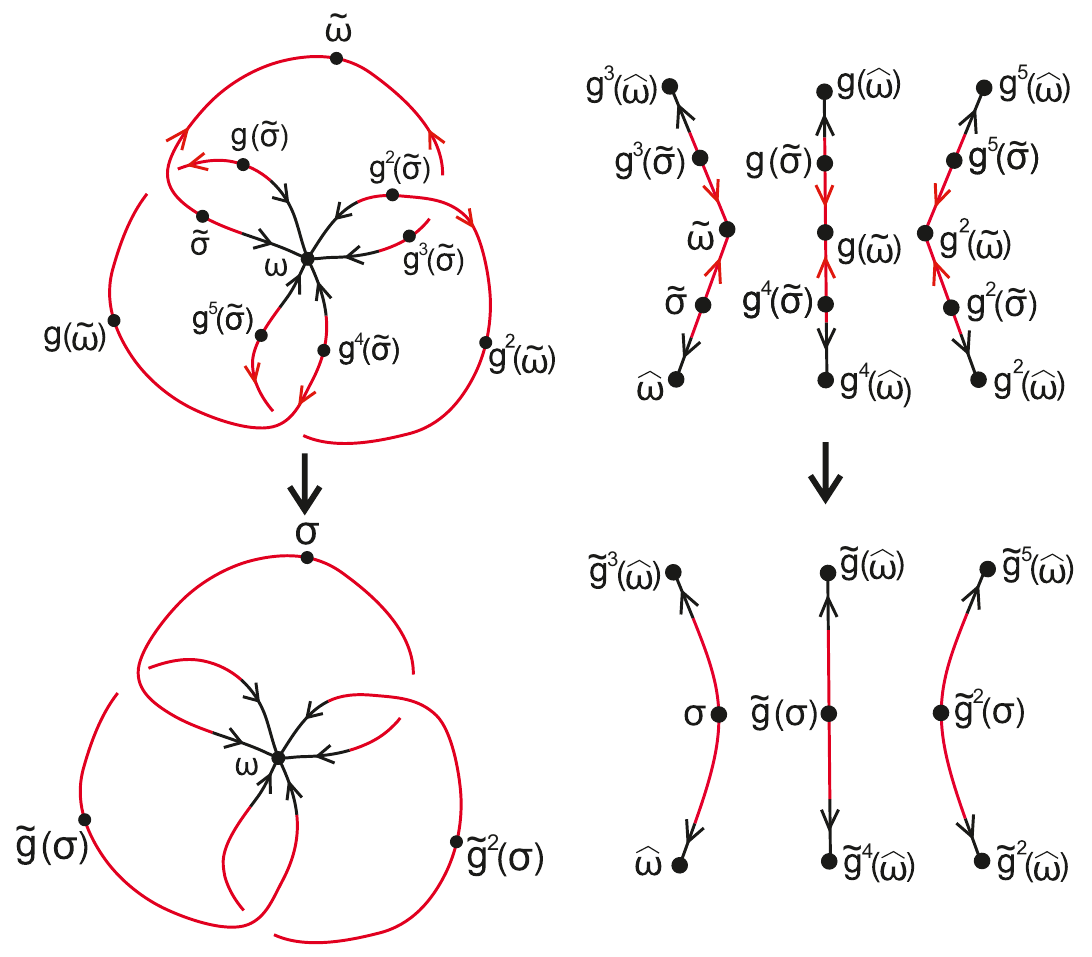}}		
			\caption{\small Illustration for case 2)}\label{case2}	
		\end{figure}
		
		Since, by Lemma \ref{g10}, the simplest diffeomorphisms in the class $G^2_4$ are pairwise topologically conjugate, any two of them are connected by a stable arc (for details, see \cite{np2021}), which completes the proof.
	\end{proof}
	
\end{document}